\documentclass[11pt]{article} 
\usepackage{hyperref}
\usepackage{textcomp}
\usepackage{amssymb,amsmath}
\usepackage{amsthm}
\begingroup
    \makeatletter
    \@for\theoremstyle:=definition,remark,plain\do{%
        \expandafter\g@addto@macro\csname th@\theoremstyle\endcsname{%
            \addtolength\thm@preskip\parskip
            }%
        }
\endgroup
\usepackage[utf8]{inputenc} 

\usepackage{geometry} 
\usepackage{graphicx} 
\usepackage{float}
\usepackage[parfill]{parskip} 

\usepackage{booktabs} 
\usepackage{array} 
\usepackage{paralist} 
\usepackage{verbatim} 
\usepackage{subfig} 

\usepackage{fancyhdr} 
\usepackage{sectsty}
\allsectionsfont{\sffamily\mdseries\upshape} 

\usepackage[nottoc,notlof,notlot]{tocbibind} 
\usepackage[titles,subfigure]{tocloft} 

\newtheorem{thm}{Theorem}[section]
\newtheorem{prop}[thm]{Proposition}
\newtheorem{lem}[thm]{Lemma}
\newtheorem{cor}[thm]{Corollary}

\newtheorem{defn}[thm]{Definition}

\newtheorem{examp}[thm]{Example}

\newcommand{\R}{\mathbb{R}}

\newcommand{\N}{\mathbb{N}}

\newcommand{\s}{\mathcal{S}}

\newcounter{qcounter}

\DeclareMathOperator{\trace}{trace}

\DeclareMathOperator{\conv}{conv}

\title{Spectrahedra and Convex Hulls of Rank-One Elements}
\author{Martin Ames Harrison}
\date{} 

\begin{document}
\clearpage\maketitle
\thispagestyle{plain}
\begin{abstract}
\noindent The Helton-Nie Conjecture (HNC) is the proposition that every convex semialgebraic set is a spectrahedral shadow. Here we prove that HNC is equivalent to another proposition related to quadratically constrained quadratic programming. Namely, that the convex hull of the rank-one elements of any spectrahedron is a spectrahedral shadow. In the case of compact convex semialgebraic sets, the spectrahedra may be taken to be compact. We illustrate the relationship between spetrahedra and these convex subsets with examples. 
\end{abstract}
\section{Introduction}
Helton and Nie conjectured in \cite{HeltSuff} that every convex semialgebraic set is a spectrahedral shadow. Recent work in convex algebraic geometry suggests that the conjecture is true. For example, Lasserre showed an approximation theorem for compact basic semialgebraic sets with nonempty interior (see \cite{LassPolar}), as well as exact semidefinite representation for sets whose defining polynomials satisfy a certain property. Scheiderer, in \cite{ClausHNC}, showed that HNC holds for convex semialgebraic subsets of $\R^2$. Others have obtained results on spectrahedral shadows themselves. We will use, for instance, the semidefinite representability (i.e. realizability as a spectrahedral shadow) of convex hulls of finite unions of spectrahedral shadows demonstrated in \cite{NetzerSinn}. Other results pertain to the boundaries and extreme points of spectrahedral shadows (e.g. \cite{NetzPlaumSchw},\cite{Rainer2}).  

In \cite{NieQC}, Nie explored the semidefinite representability of quadratically parametrized images of quadratically constrained sets using results from the theory of moments. Nie showed that when only one quadratic constraint is imposed, or when two quadratic constraints of a certain type are imposed, it follows that such an image must be a spectrahedral shadow. Even when many constraints are allowed, this may appear to be a highly specialized question. In fact, we will show that it is equivalent to HNC, and can be expressed in terms of natural subsets of spectrahedra. 

We end this section with some definitions. The set of $N\times N$ matrices over $\R$ is denoted by $M_N(\R)$. The set $\s_N\subseteq M_N(\R)$ consists of the symmetric matrices. If $A\in\s_N$ has only nonnegative eigenvalues, then we write $A \in \s_N^+$ and $A \succeq 0$, and we call $A$ \emph{positive semidefinite} (PSD). If $A \in \s_N^+$ has only positive eigenvalues, then we call $A$ \emph{positive definite} and write $A \succ 0$. Positive definite matrices comprise the interior of $\s_N^+$, which is denoted by $\s_N^{++}$. Equivalently, $A$ is positive definite if and only if $x^TAx >0$ for all $x \neq 0$. 

See \cite{Barvinok} for an in-depth exposition on the PSD cone. 

\begin{defn}
A {\bf spectrahedron} is a set of the form $\s_N^+ \cap V$, where $V$ is an affine subspace of $\s_N$. 
\end{defn}

\begin{defn}
The terms {\bf projected spectrahedron, spectrahedral shadow,} and {\bf semidefinitely representable set} all denote the image under a linear transformation of a spectrahedron.
\end{defn}

For convenience, we invent a term for the convex hulls of interest.

\begin{defn}
A {\bf pseudospectrahedron} is the convex hull of all rank-1 elements of a spectrahedron.  
\end{defn}

And we give a name to the proposition to be proven equivalent to HNC. 

\begin{prop}{\bf (\emph{Pseudospectrahedron Conjecture -- PSC})} Every pseudospectrahedron is a spectrahedral shadow.
\end{prop}

And, for the sake of clarity and completeness, below is at statement of the Helton-Nie Conjecture.

\begin{prop}{\bf (\emph{Helton-Nie Conjecture -- HNC})} Every convex semialgebraic set is a spectrahedral shadow.
\end{prop}

\section{Results}

The following is not too hard to prove, and similar claims with essentially the same proof appear elsewhere. But it is important to establish that we are not presenting a strictly stronger claim whose falsehood would have no bearing on HNC. 

\begin{prop}
$\text{HNC} \Rightarrow \text{PSC}$
\label{HNCPNC}
\end{prop}

\begin{proof}
Suppose that HNC is true. That is, suppose that every convex semialgebraic set is a spectrahedral shadow. Let $P$ be a pseudospectrahedron in $\s_N^+$ defined by equations $\langle A_i, X \rangle = a_i$, for $i$ satisfying $1\leq i \leq k$. Then $P$ is the convex hull of a semialgebraic set, and from this it follows that $P$ is itself a semialgebraic set; combining the ``Projection Theorem" found in \cite{PRSTL} with the argument below completes the proof.

In general, given a set $S \subset \R^n$, Carath\'eodory's theorem lets us write the convex hull of $S$ as the projection of a set in higher dimension. To see how, first write
\[S' = \{(x,y_1,\ldots,y_{n+1},t) \mid x \in \R^n,y_i \in S, t \in \Delta^n, x= \sum t_iy_i \}, \]
where $\Delta^n$ is the standard simplex $\{t \in \R^{n+1}\mid t_i \geq 0, \sum t_i = 1\}$. Then, letting $\Pi_n$ denote the projection onto the first $n$  coordinates of the space $\R^{n+n(n+1)+(n+1)}$ in which $S'$ is situated, we have
\[\conv S        = \Pi_n(S').\]

The point of this general observation is that when $S$ is semialgebraic, then so is $S'$. Therefore this argument shows that the convex hull of a semialgebraic set is a projection of a semialgebraic set. Therefore the arbitrary $P$ named above is a projection of a semialgebraic set. Therefore $P$ is a convex semialgebraic set. Assuming HNC, we conclude that any such $P$ must be a spectrahedral shadow.
\end{proof}

The converse requires a technical lemma. We will state it as simply as possible and then give a full interpretation as we prove it.

\begin{lem}
Any system of polynomial equations and ineqaulities is equivalent to a system of quadratic equations and inequalities.
\label{ReduceQuadratic}
\end{lem}
This is a well-known fact and appears in several different forms. We will use only the version above.
\begin{proof}
Let $f_1,\ldots,f_k \in \R[x_1,\ldots,x_n]$, where $\R[x_1,\ldots,x_n]$ is the space of polynomials over $\R$ in the indeterminates $x_1,\ldots,x_n$.  We will introduce new indeterminates and equations to the system $F(x)=\wedge_i (f_i(x)\boxdot_i 0)$,where $\boxdot_i \in \{=,<,>,\geq,\leq\}$, in such a way that the projection of the resulting system's feasible region onto the original $n$ coordinates is exactly the same as the feasible region $R(F)=\{x \in \R^n \mid F(x)\}$ of $F$. We will also ensure that this new augmented system, denoted by $\tilde{F}$, has no expressions of degree more than $2$. To start, we identify each monomial of degree at least $2$ with a new indeterminate, and append to $F$ a certain set of the equations these new indeterminates must satisfy. Let $u_1,\ldots, u_M$ be these new indeterminates.

For each $a \in \{1,\ldots,M\}$  we form the set $E_a$ consisting of exactly the quadratic polynomials $u_a - x_bv \in \R[x_1,\ldots,x_n,u_1,\ldots,u_M]$ for which $x_b$ divides $u_a$ and $v$ is the unique element of $\{x_1,\ldots,x_n,u_1,\ldots,u_M\}$ satisfying $u_a = x_bv$. Note that the size of $E_a$ will be exactly the number of indeterminates $x_b$ which divide $u_a$, and that $E_a\cap E_b = \emptyset$ exactly when $a\neq b$. Taking the union of the $E_a$ over all $a \in \{1,\ldots,M\}$, setting each element equal to $0$, and appending these equations to the original system, we nearly have $\tilde{F}$; to complete it, we replace all monomials of degree $3$ or more appearing in the $f_i$ with quadratic terms $x_bu_c$. There is room for choice in doing so, but no inconsistency can arise from making different choices for the same monomial in different $f_i$.  This is $\tilde{F}$, and it contains only quadratic polynomials; we write it explicitly as
\[\tilde{F}(x,u) =   (\wedge_i (\tilde{f_i}(x,u)\boxdot_i 0)) \wedge (\wedge_j(q_j(x,u) = 0)),           \]
where $\tilde{f_i}(x,u) \in \R[x_1,\ldots,x_n,u_1,\ldots,u_M]$ is a polynomial obtained from $f_i(x)$ in the manner just described, and $q_j \in \R[x_1,\ldots,x_n,u_1,\ldots,u_M]$ are the elements of the sets $E_1, \ldots, E_M$. The feasible region of $\tilde{F}$ is denoted by $R(\tilde{F}) = \{(x,u) \in \R^{n+M} \mid \tilde{F}(x,u)\}$

Let $\Pi_n$ be as above, and suppose that $p\in \R^n$ is such that $F(p)$. By construction of the sets $E_a$, the point $(p,u_1(p),\ldots,u_M(p)) \in \R^{n+M}$ must satisfy $\tilde{F}$; that the $u_a$ can be uniquely expressed as monomial functions of $x = (x_1,\ldots,x_n)$ can be established by a routine induction on the degree of $u_a$.  Since $p = \Pi_n(p,u_1(p),\ldots,u_M(p))$, we have shown 
\[R(F) \subseteq \Pi_n(R(\tilde{F}))\]
Conversely, suppose that $(p,w) \in \R^{n+M}$ is such that $\tilde{F}(p,w)$. Then we may use the identities $u_a - x_bu_c = 0$ to eliminate all $u_a$, for $a = 1,\ldots,M$, from the polynomials $\tilde{f_j}(x,u)$. This recovers the original $f_i$ and yields the proposition $F(p)$. Thus we have shown
\[R(F) \supseteq \Pi_n(R(\tilde{F})),\]
and therefore
\[R(F) = \Pi_n(R(\tilde{F})).\]

\end{proof}

We will use the above result to show how any semialgebraic set can be realized as the projection of a set defined by quadratic equations and inequalities, or a finite union thereof. From there we can show that every convex semialgebraic set is convex hull of projections of finitely many pseudospectrahedra. If every pseudospectrahedron is in turn the projection of a spectrahedron (PSC), then we can conclude that HNC is true. Thus it follows that PSC implies HNC. Before we proceed to the proof of \ref{PSCHNC} we cite another technical lemma due to Netzer and Sinn:

\begin{lem}The convex hull of a union of finitely many spectrahedral shadows is itself a spectrahedral shadow.
\label{NetzerS}
\end{lem}
See the brief article \cite{NetzerSinn} for a proof.

\begin{prop}
$\text{HNC} \Leftarrow \text{PSC}$
\label{PSCHNC}
\end{prop}

\begin{proof}
Assume PSC. Let $S \subseteq \R^n$ be a convex semialgebraic set. Since $S$ is semialgebraic, it can be realized as the finite union of sets defined by conjuctions of finitely many polynomial equations and inequalities. Since $S$ is convex, it is equal to its own convex hull. Since the convex hull of any set must contain the convex hull of any of its subsets, \ref{NetzerS} lets us assume that $S$ is determined by a conjunction of finitely many polynomial equations and inequalities. In this step, however, we lose the assumption of convexity; we now have the task of showing that the convex hull of any such $S$ is a spectrahedral shadow. By \ref{ReduceQuadratic}, we may assume that $S$ is defined by polynomials of degree $2$ or less. With these assumptions, we now show how to express $\conv( S)$ as the projection of a pseudospectrahedron.

Suppose that $S$ is the set of all $x=(x_1,\ldots,x_n) \in \R^n$ satisfying $f_a(1,x) > 0$ and $q_a(1,x) = 0$ for all $a \in \{ 1,\ldots,k\}$ (we may assume there are as many equations as inequalities by including trivial expressions if necessary), where $f_a,q_a \in \R[x_0,x_1,\ldots,x_n]$ are homogeneous polynomials given by
\[f_a(x_0,\ldots,x_n) =\sum_{i=0}^{n} \sum_{j\leq i} f_{a,i,j} x_ix_j, \text{ and }q_a(x_0,\ldots,x_n) =\sum_{i=0}^{n} \sum_{j\leq i} q_{a,i,j} x_ix_j. \]
Then $S$ is the projection in $\R^n$ of the set of all points $(x,y,z,r) \in \R^n\times \R^k \times \R^k \times \R^k$ satisfying the equations
\[ q_a(1,x) =0, f_a(1,x)=y_a, y_a- r_a^2=0, \text{ and } y_az_a=1 \text{ for all } a \in \{1,\ldots,k\}.\]
Note that we have expressed $S$ as the projection of a set described by only \emph{equations} in polynomials of degree at most $2$. This is a crucial step because inequalites cannot be encoded directly as the affine constraints of a spectrahedron. Note also that the possible values for $y_a\in \R$ are exactly the positive real numbers.

Let $N=3k+n+1$. We now define the matrices which enable us to translate our polynomial constraints into affine constraints on the cone $\s_N^+$. For each subset $\alpha=\{\alpha_1,\alpha_2\} \subseteq \{0,\ldots,n\}$ of size $2$, define the matrix $A_\alpha \in \s_N$ by

\[ (A_\alpha)_{i,j} =
 \begin{cases}
1\slash 2, & \text{if }\{i-1,j-1\}=\alpha \\
0, & \text{otherwise }
\end{cases}.\]
For every $b \in \{0,\ldots,n\}$, define
\[ (A_b)_{i,j} =
 \begin{cases}
1, & \text{if } i = j = b+1 \\
0, & \text{otherwise }
\end{cases}.\]
For every $a \in \{1,\ldots,k\}$ define the matrix $Y_a$ by
\[ (Y_a)_{i,j} =
 \begin{cases}
1\slash 2, & \text{if } \{i , j\} =\{0, n+a+1\} \\
0, & \text{otherwise }
\end{cases},\]
the matrix $g_a$ by
\[ (g_a)_{i,j} =
 \begin{cases}
1, & \text{if } i = j =n+2k+a+1  \\
0, & \text{otherwise }
\end{cases},\]
and the matrix $YZ_a$ by
\[ (YZ_a)_{i,j} =
 \begin{cases}
1\slash 2, & \text{if } \{i,j\}=\{n+a+1,n+k+a+1\}  \\
0, & \text{otherwise }
\end{cases}.\]
For each $a \in \{1,\ldots,k\}$, we define the matrix $F_a$ by
\[F_a =\sum_i \sum_{j\leq i} f_{a,i,j} A_{\{i,j\}}-Y_a,\]
the matrix $Q_a$ by 
\[Q_a = \sum_i \sum_{j\leq i} q_{a,i,j} A_{\{i,j\}}, \]
and the matrix $G_a$ by
\[G_a = Y_a - g_a.\]

Finally, we are ready to define the pseudospectrahedron whose projection in $\R^n$ is equal to $\conv(S)$:

Let $P$ denote the set of all $X \in \s_N^+$ of rank $1$ satisfying, for all $a \in \{1,\ldots,k\}$, the equations $\langle X, F_a\rangle=\langle X,Q_a\rangle =\langle X,G_a\rangle=0$, and $\langle X,A_0\rangle = \langle X, YZ_a \rangle = 1$. 

Then $S' \equiv \conv(P)$ is the desired pseudospectrahedron. To prove this, we will show that $S = \Pi_n(P)$, where

 \[\Pi_n : X \mapsto (\langle X, A_{\{0,1\}}\rangle,\ldots,\langle X, A_{\{0,n\}}\rangle).\] 

Since $\Pi_n$ is linear, the equality is preserved under taking convex hulls so that we will obtain $\conv(S) = \Pi_n(\conv(P))= \Pi_n(S')$.

The elements $X \in \s_N^+$ of rank $1$ are exactly the matrices $v^Tv$ for any $v$ of the form

\[(v_1,\ldots,v_N).\]

If $X = v^Tv$ is such an element, then the constraint $\langle X, A_0 \rangle = 1$ is satisfied if and only if $v_1^2 = 1$. Assuming this, we find the point $(v_1^2,v_1v_2,\ldots,v_1v_{n+1})$ satisfies the relations defining $S$ if and only if $\langle X, F_a\rangle=\langle X,Q_a\rangle =\langle X,G_a\rangle=0$, and $\langle X,A_0\rangle = \langle X, YZ_a \rangle = 1$ are satisfied for all $a \in \{1,\ldots,k\}$. In particular, 
\[ \langle X,F_a \rangle=0 \Rightarrow f_a(1,v_1v_2,\ldots,v_1v_{n+1}) >0,\]
and
\[\langle X,Q_a\rangle=0 \Rightarrow q_a(1,v_1v_2,\ldots,v_1v_{n+1}) =0\]
for all $a \in \{1,\ldots,k\}$.
This shows that $\Pi_n(P) \subseteq S$.

Conversely, suppose that $x=(x_1,\ldots,x_n)$ satisfies $f_a(1,x)>0$ and $q_a(1,x)=0$ for all $a \in \{1,\ldots,k\}$. Set $v_1 = 1$ and $v_j = x_{j-1}$ for $j \in \{2,\ldots,n+1\}$. For $j \in \{n+2,\ldots,n+k+1\}$, set $v_j =f_{j-(n+1)}(1,x)$, $v_{j+k} = 1\slash v_j$, and $v_{j+2k} = \sqrt{v_j}$. If $v = (v_1,\ldots,v_N)$, then $X=v^Tv$ belongs to $P$ and $\Pi_n(X) = x$. This shows that $S \subseteq \Pi_n(P)$, and therefore $S= \Pi_n(P)$.

As discussed above, it follows that $\conv(S) = \Pi_n(S')$. Assuming PSC, we note that $S'$ is a projection of a spectrahedron, and therefore any projection of $S'$ is also a spectrahedral shadow. Thus HNC follows from PSC.
\end{proof}
We now modify the above argument to show that if PSC is true for all compact spectrahedra, then HNC is true for all compact convex semialgebraic sets.

\begin{prop}{\bf(\emph {Compact PSC $\Leftrightarrow$ Compact HNC})}
All compact convex semialgebraic sets are spectrahedral shadows if and only if all compact pseudospectrahedra are spectrahedral shadows.
\label{compactPSC}
\end{prop}
\begin{proof}
Let $C$ be a compact convex semialgebraic set in $\R^n$. By Theorem $3.5$ in \cite{HeltSuff}, $C$ can be realized as a union of finitely many sets of the form

\[ \{x \in \R^n \mid f_i(x) \geq 0, i = 1,\ldots,m\},\]

where each $f_i$ is a polynomial and $m \in \N$. Fix $S = \{x \in \R^n \mid f_i(x) \geq 0, i = 1,\ldots,m\}$.

By introducing new indeterminates and quadratic equations as above, we replace each inequality $f_i(x) \geq 0$ with an equation $F_i(1,x,y) = z_i^2$, where the $F_i$ are quadratic and $x_0,y = y_1,\ldots,y_k,z = z_1,\ldots,z_m$ are the new indeterminates. Since each $f_i$ is a polynomial, these equations imply that the values of $z_i$ are bounded (by the compactness of $S$). We take $w$ to be a slack variable, and impose the constraint
\[\|(x_0,x,y,z,w)\|^2 = B,\] 

for an appropriately large number $B>1$ (since $x_0^2 = 1$) -- a value so large that this constraint leaves $S$ unchanged. This ensures that the spectrahedron associated to $S$ lies in the affine plane of matrices having trace equal to $B$. 

If the pseudospectrahedron so obtained is a spectrahedral shadow, then $C$ is a spectrahedral shadow. Thus PSC for compact sets implies HNC for compact sets.
\end{proof}
Finally, we show that compact quadratic programming can be exactly and uniformly encoded as semidefinite programmming if and only if HNC holds for compact convex semialgebraic sets. 
\begin{cor}
Propositions $i$ and $ii$ below are equivalent.
\begin{list}{\roman{qcounter})}{\usecounter{qcounter}}
\item Compact convex semialgebraic sets are spectrahedral shadows. 
\item If $C\subset \R^n$ is compact and defined by finitely many quadratic equalities, then for some $N \in \N$ there is a compact spectrahedron $S\subseteq \s_N$ and a linear transformation $T: \s_N \to \s_n$ satisfying 
\[ \inf\{ x^TAx \mid x \in C \} = \inf \{\langle TX, A \rangle \mid X \in S\} \]
for all $A \in \s_ n$.
\end{list}
\end{cor}
\begin{proof}
If $i$ holds, then by \ref{compactPSC} we may express the convex hull of $\{xx^T \mid x \in  C\}$ as a spectrahedral shadow $T(S)$ where $S$ is compact, and $ii$ is immediate.

Conversely, suppose that $ii$ is given, and let $C$ be a compact pseudospectrahedron in $\s_n$. By \ref{compactPSC}, it is enough to prove that $C$ is a spectrahedral shadow. Let $S$ be a compact spectrahedron and $T$ a linear transformation such that
\[ \inf\{ \langle X, A\rangle \mid X \in C \} = \inf \{\langle TX, A \rangle \mid X \in S\} \]
for all $A \in \s_ n$. If necessary, we can translate both $C$ and $T(S)$ and intersect them with an affine plane to ensure that they meet the hypotheses of the ``Bipolar Theorem" as stated in \cite{Barvinok}; none of this will alter the above equality. By compactness, both of $C$ and $T(S)$ are closed. We may therefore conclude that $C =(C^\circ)^\circ=(T(S)^\circ)^\circ = T(S)$.
\end{proof}

\section{Examples}
We finish with two examples. The first illustrates what happens in the case of a single positive definite constraint -- the spectrahedron and its pseudospectrahedron coincide. The second shows that this is not always the case. 

\begin{examp}
Define the set
\[S \equiv \{X\in \s_3^+ \mid \trace X = 1\}. \]
Then $S$ is isomorphic to the set 
\[\{(x,y) \mid x(1-x)-y^2 \geq 0\},\]
which corresponds to the planar region bounded by the green curve in figure \ref{2d} below. Notice that every element of the boundary of $S$ has rank equal to one. The spectrahedron $S$ is therefore the convex hull of its rank-one elements.
\begin{figure}[H]

   \begin{center}
    \includegraphics[width=0.5\textwidth]{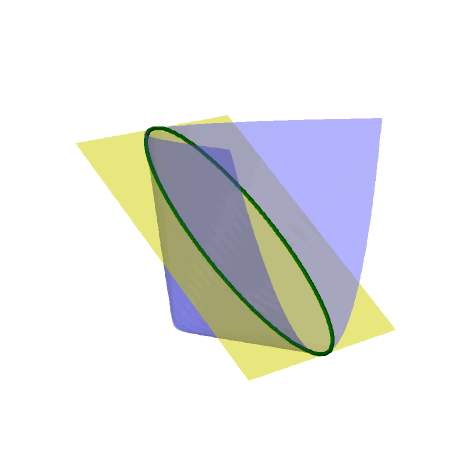}
    \caption{The green curve consists of exactly the rank-one points of this spectrahedron.}
    \label{2d}
  \end{center}
\end{figure}
This property is independent of the low dimension of our example; it is true in general that the PSD matrices with trace equal to one comprise a pseudospectrahedron.
\end{examp}

Below is a slightly less trivial case. This time, we find that the given pseudospectrahedron is in fact a polytope.

\begin{examp}
Define $S$ to be the spectrahedron 
\[\Bigg\{ (x,y,z) \mid \begin{pmatrix}x+1&z&y\\z&1-y&x\\y&x&z+1\end{pmatrix} \succeq 0 \Bigg\}.\]
Then $S$ is a compact spectrahedron. Solving for the rank-one points of $S$, we find exactly four. These four points are the vertices of a tetrahedron, and the line segments connecting one to another lie in the boundary of $S$. These points and segments are depicted in green and yellow, respectively, in figure \ref{3d}. This is not a surprise; for any spectrahedron $S \subseteq \s_n^{+}$ which meets $\s_n^{++}$, all convex combinations of $n-1$ or fewer rank-one elements of $S$ must lie on the relative boundary of $S$.
\begin{figure}[H]

   \begin{center}
    \includegraphics[width=0.5\textwidth]{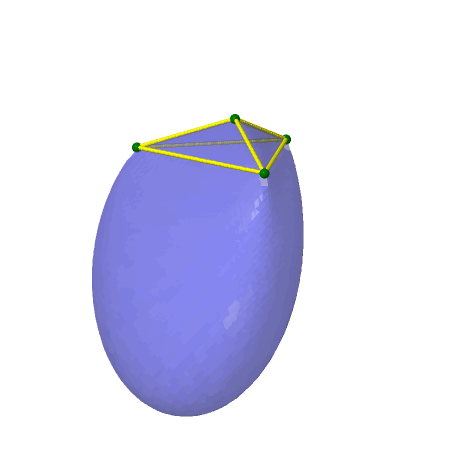}
    \caption{The four green points are exactly those of rank one. The yellow line segments lie on the boundary of the spectrahedron $S$ and their relative interior points all have rank two.}
    \label{3d}
  \end{center}
\end{figure}
\end{examp}
The graphics in the preceding examples were created using Sage (see \cite{SAGE}).
\bibliographystyle{plain}
\bibliography{./mainbib}

\end{document}